\documentclass[11pt]{amsart}
\usepackage{palatino} 
\usepackage{amssymb,amsmath,url,graphicx}
\usepackage[all]{xy}
\usepackage{enumitem}
\usepackage{graphicx}
\usepackage[hidelinks]{hyperref}
\usepackage{fullpage}
\usepackage{color}

\newif\ifdebug                                                      %
\debugfalse

\newif\ifinfootnote
\infootnotefalse

\let\footnoteasusual\footnote
\renewcommand{\footnote}[1]
{\infootnotetrue\footnoteasusual{#1}\infootnotefalse}

\newcommand{\printname}[1]
{\ifmmode{ \smash{ \raisebox{5pt}{\text{\tiny{#1}}} } }
 \else   {\ifinfootnote \smash{\raisebox{0pt}{\tiny{#1}}}
             \else { \marginpar{
                     \smash{ \makebox[0pt]{\raisebox{-12pt}{\tiny{#1}}} }
                               } } \fi} \fi}

\swapnumbers
\numberwithin{equation}{section}
\newtheorem {Theorem}[equation]         {Theorem}

\newtheorem {Lemma}[equation]           {Lemma}

\newtheorem {Claim*}                    {Claim}

\newtheorem {Corollary} [equation]      {Corollary}
\newtheorem {Proposition}  [equation]   {Proposition}

\theoremstyle{definition}
\newtheorem{Definition}[equation]{Definition}

\theoremstyle{remark}
\newtheorem{Remark}[equation]{Remark}
\newtheorem*{Remark*}{Remark}
\newtheorem{Example}[equation]{Example}

\setlist{topsep=0pt,itemsep=6pt}

\def \Z {{\mathbb Z}}

\def \C {{\mathbb C}}

\newcommand{\cs}{\mathcal{S}}
\newcommand{\cw}{\mathcal{W}}

\newcommand{\Hess}{{\mathcal{H}\textup{ess}}}

\newcommand{\Symm}{{\mathfrak{S}}}

\title[Title]{Torus fixed point sets of Hessenberg Schubert varieties in regular semisimple Hessenberg varieties}

\author{Megumi Harada}
\address{Dept.\ of Mathematics and Statistics, McMaster University, 
1280 Main Street West, 
Hamilton, Ontario L8S 4K1, Canada}
\email{Megumi.Harada@math.mcmaster.ca}

\author{Martha Precup}
\address{Department of Mathematics and Statistics\\ Washington University in St. Louis \\ One Brookings Drive \\ St. Louis, Missouri  63130 \\ U.S.A. }
\email{martha.precup@wustl.edu}

\date{\today}

\keywords{Hessenberg variety, flag variety, Schubert variety, Schubert cells, Bruhat order, acyclic orientations, reachability, subsets of Weyl type}

\begin{document}

\maketitle

\begin{abstract} 
It is well-known that the $T$-fixed points of a Schubert variety in the flag variety $GL_n(\C)/B$ can be characterized purely combinatorially in terms of Bruhat order on the symmetric group $\Symm_n$.  
In a recent preprint, Cho, Hong, and Lee give a combinatorial description of the $T$-fixed points of Hessenberg analogues of Schubert varieties (which we call Hessenberg Schubert varieties) in a regular semisimple Hessenberg variety.  This note gives an interpretation of their result in terms of Bruhat order by making use of a partition of the symmetric group defined using so-called subsets of Weyl type.  The Appendix, written by Michael Zeng, proves a lemma concerning subsets of Weyl type which is required in our arguments.   
\end{abstract}

\section{Introduction} 

The main result of this note is  a characterization of the $T$-fixed point sets of (opposite) Hessenberg Schubert varieties in terms of Bruhat order.  We achieve this by giving an interpretation in terms of Bruhat order of the notion of \textit{reachability} -- a concept introduced by Cho, Hong, and Lee in~\cite{Cho-Hong-Lee1}. 
Hessenberg varieties has recently garnered great interest in different research communities due to their connections (which have come to light in the past decade) to many areas, not the least of which is the famously unsolved Stanley--Stembridge conjecture in algebraic combinatorics~\cite{ShareshianWachs2016}; see~\cite{AbeHoriguchi} for an overview. 
 In \cite{Cho-Hong-Lee1}, the authors analyze the Bia\l{}ynicki--Birula decomposition of a regular semisimple Hessenberg variety, and they introduce and use the notion of reachability to give an explicit description of the $T$-fixed points of the closure of a given Bia\l{}ynicki-Birula cell. Such a closure is called a {Hessenberg Schubert variety}, since these are analogues of classical Schubert varieties in the flag variety.

 To describe our results, we briefly recall some terminology. Let $w\in \Symm_n$ and $\Omega_w^\circ := B_-wB/B$ denote the corresponding opposite Schubert cell in the type A flag variety $GL_n(\C)/B$.  Here $B$ and $B_-$ denote the Borel subgroups in $GL_n(\C)$ of upper-triangular and lower-triangular matrices, respectively. Given $u\in \Symm_n$, the partial order on $\Symm_n$ defined by 
\begin{eqnarray}\label{eqn.Bruhat}
w\leq u \; \textup{ whenever } \; uB\in \Omega_w:= \overline{\Omega_w^\circ}
\end{eqnarray}
is called \textbf{Bruhat order}.  This order is fundamental in the study of the symmetric group and the geometry of flag varieties and related spaces.
Now let $\mathsf{S} \in \mathfrak{gl}_n(\C)$ be a diagonal matrix with distinct eigenvalues, $h: [n] \to [n]$ a Hessenberg function, and $\Hess(\mathsf{S}, h)$ the corresponding regular semisimple Hessenberg variety in $GL_n(\C)/B$ (for definitions see Section~\ref{sec: background}).  We call the intersection $\Omega_{w,h}^\circ = \Omega_{w}^\circ\cap \Hess(\mathsf{S},h)$ the  (opposite) Hessenberg Schubert cell indexed by $w \in \Symm_n$, and its closure $\Omega_{w,h}:=\overline{\Omega_{w,h}^\circ}$ is the 
\textbf{opposite Hessenberg Schubert variety indexed by $w$}.  
 The torus $T:=B\cap B_-$ of diagonal matrices in $GL_n(\C)$ acts on $\Hess(\mathsf{S}, h)$ by left multiplication.
Cho, Hong, and Lee gave an explicit characterization of the $T$-fixed points in $\Omega_{w,h}$ in terms of reachability \cite{Cho-Hong-Lee1}. 
Our main contribution in this note is a reinterpretation, using subsets of Weyl type and the corresponding partition of $\Symm_n$, of reachability in terms of Bruhat order (Proposition~\ref{prop.reachability}), which then allows us to give a concrete description of the $T$-fixed points in $\Omega_{w,h}$ in terms of Bruhat order (Theorem~\ref{theorem: main} and Corollary~\ref{cor.reformulation}).

As is implied above, our main tool is a partition of $\Symm_n$ into sets $\cw(\cs,h)$ defined by particular subsets $\cs$ in the type A root system, called {subsets of Weyl type (with respect to $h$)}; cf.~Definition~\ref{definition: subset Weyl type} below.  This partition was introduced by Sommers--Tymoczko in~\cite{SomTym06} and subsequently used by the second author in~\cite{Precup2018} to prove that the Betti numbers of regular Hessenberg varieties are palindromic in all Lie types.   Appendix~\ref{sec.appendix} below by Michael Zeng proves that each set $\cw(\cs,h)$ is a weak Bruhat interval.  Let $w_\cs\in \cw(\cs,h)$ denote the maximal element in this interval.  Theorem~\ref{theorem: main} below shows that $\Omega_{w_\cs,h}^T = \Omega_{w_\cs}^T$.  Thus the notion of reachability, which defines a Hessenberg variation of the partial order in~\eqref{eqn.Bruhat}, can be characterized using Bruhat order.

 Example~\ref{key-ex} below shows that $\Omega_{w_\cs, h}$ is not simply a union of the Hessenberg Schubert cells indexed by $u\geq w_\cs$.  In other words, our Theorem~\ref{theorem: main} does not yield a complete description of the Hessenberg Schubert variety $\Omega_{w_\cs, h}$.

We expect that our interpretation of the results of Cho--Hong--Lee in the language of subsets of Weyl type and Bruhat order will yield further insights into the geometry and combinatorics of Hessenberg varieties.  We hope that this interpretation will shed light on the still-unsolved problem of fully characterizing the closures of (opposite) Hessenberg Schubert cells.  
We leave this problem for future work.

\smallskip
\textbf{Acknowledgements:}  
This work was supported in part by the National Security Agency
under Grant No.~H98230-19-1-0119, The Lyda Hill Foundation, The McGovern
Foundation, and Microsoft Research, as part of the Mathematical Sciences Research Institute Summer Research for Women program. 
The first author is supported by a Natural Science and Engineering Research Council Discovery Grant and a Canada Research Chair (Tier 2) from the Government of Canada. The second author is supported in part by NSF DMS-1954001. 
The Appendix is by Michael Zeng, and is part of an undergraduate research project conducted under the second author's supervision in 2020.

\section{Background}\label{sec: background} 

\subsection{Hessenberg Varieties and Hessenberg Schubert cells}

Hessenberg varieties in Lie type A are subvarieties of the (full) flag variety $GL_n(\C)/B$ where $B$ is the Borel subgroup of upper triangular matrices in $GL_n(\C)$.   Let $G = GL_n(\C)$ and let $B_-$ denote the Borel subgroup of lower triangular matrices. 
The following two cell decompositions of $G/B$ (both called a Bruhat decomposition of $G/B$) are well-studied: 

\begin{eqnarray}\label{eqn.Bruhat.decomp}
G/B = \bigsqcup_{w\in \Symm_n} X_w^\circ = \bigsqcup_{w\in \Symm_n} \Omega_{w}^\circ
\end{eqnarray}
where $X_w^\circ:= BwB/B$ is the \textbf{Schubert cell} and $\Omega_w^\circ = B_-wB/B$ is the \textbf{opposite Schubert cell}.  The closure $X_w:= \overline{X_w^\circ}$ (respectively $\Omega_w:= \overline{\Omega_w^\circ}$) is called the \textbf{Schubert variety} (respectively \textbf{opposite Schubert variety}) for $w\in \Symm_n$.  It is an important and well-known fact that
\begin{eqnarray} \label{eqn.Schubert.varieties}
X_w = \bigsqcup_{u\leq w} X_u^{\circ} \; \textup{ and } \; \Omega_w = \bigsqcup_{u\geq w} \Omega_u^{\circ}
\end{eqnarray}
where $\leq$ denotes the Bruhat order on $\Symm_n$ defined in~\eqref{eqn.Bruhat}.  

We denote the root system of $\mathfrak{gl}_n(\C)$ by $ \Phi = \{ t_i-t_j \mid 1\leq i \neq j \leq n \}$ with positive roots $\Phi^+ = \{ t_i-t_j\in \Phi \mid i<j \}$, negative roots $\Phi^- =  \{ t_i-t_j\in \Phi \mid i>j \}$, and simple positive roots $\Delta = \{t_i-t_{i+1} \mid 1\leq i \leq n-1\}$.   Given $w\in \Symm_n$ the \textbf{inversion set} of $w$ is 
\[
N(w):= \{ t_i-t_j \in \Phi^+ \mid w(t_i-t_j) \in \Phi^- \} = \Phi^+ \cap w^{-1}(\Phi^-).
\]
We set $\ell(w):= |N(w)| = |\{ i<j \mid w(i)>w(j) \}|$.  It is known that $X_w^\circ \simeq \C^{\ell(w)}$ and $\Omega_w^\circ\simeq \C^{N-\ell(w)}$ where $N=\sum_{i=1}^{n-1}(n-i) = \dim_\C GL_n(\C)/B$.

A Hessenberg variety in $G/B$ is specified by two pieces of data: a \textbf{Hessenberg function}, that is, a nondecreasing function $h:\{1,2,\ldots,n\} \rightarrow \{1,2,\ldots,n\}$ such that $h(i) \geq i$ for all $i$, and a choice of an element $\mathsf{X}$ in $\mathfrak{gl} (n,\C)$. We frequently write a Hessenberg function by listing its values in sequence,
i.e., $h = (h(1), h(2), \ldots, h(n))$. 
The \textbf{Hessenberg variety} associated to the linear operator $\mathsf{X}$ and Hessenberg function $h$ and is defined as
\begin{equation}\label{eq: definition Hess X h}
{\mathcal{H}ess}(\mathsf{X},h) = \{ gB \mid \mathsf{X} g_i \in \mathrm{span}_\C\{ g_1, \ldots, g_{h(i)}  \} \}
\end{equation}
where $g_1, \ldots, g_n$ denote the columns of $g\in GL_n(\C)$.
In this paper, we focus on the case when $\mathsf{X}$ is a regular semisimple operator $\mathsf{S}$ (i.e., diagonalizable with distinct eigenvalues); more specifically, we fix $\mathsf{S}$ to be a diagonal matrix with distinct eigenvalues. We refer to the corresponding Hessenberg variety $\Hess(\mathsf{S}, h)$ as a \textbf{regular semisimple Hessenberg variety}. 
It was established in~\cite{DeMProSha92} that $\Hess(\mathsf{S},h)$ is a smooth, irreducible variety of dimension $N_h = \sum_{i=1}^{n-1} (h(i)-i)$.

For each $w\in \Symm_n$ we consider the \textbf{Hessenberg Schubert cell}, defined as
\[
X_{w,h}^\circ 
= X_w^\circ\cap \Hess(\mathsf{S}, h) = BwB/B \cap \Hess(\mathsf{S},h)
\]
and also the \textbf{opposite Hessenberg Schubert cell}, defined as 
\begin{equation}\label{eq: opp Hess Schubert} 
\Omega_{w,h}^\circ 
= \Omega_w^\circ\cap \Hess(\mathsf{S}, h) = B_{-}wB/B \cap \Hess(\mathsf{S},h).
\end{equation} 
From this, we obtain decompositions of the regular semisimple Hessenberg variety
\begin{eqnarray}\label{eqn.decomposition}
\Hess(\mathsf{S}, h) = \bigsqcup_{w\in \Symm_n} X_{w,h}^\circ = \bigsqcup_{w\in \Symm_n} \Omega_{w,h}^\circ
\end{eqnarray}
where $X_{w,h}^\circ \simeq \C^{\ell_h(w)}$ and $\Omega_{w,h}^\circ \simeq \C^{N_h- \ell_h(w)}$ for
\begin{eqnarray}\label{eqn.Hess.length}
\ell_h(w):= |\{ i<j \mid w(i)>w(j) \textup{ and } j\leq h(i) \}.
\end{eqnarray}
When $h=(n,n,\ldots, n)$ then $\Hess(\mathsf{S},h) = G/B$ and we recover the Bruhat decomposition of $GL_n(\C)/B$ from~\eqref{eqn.Bruhat.decomp}.  We now define the \textbf{Hessenberg Schubert variety} for $w\in \Symm_n$ to be $X_{w,h}:= \overline{X_{w,h}^\circ}$ and the \textbf{opposite Hessenberg Schubert variety} to be  $\Omega_{w,h}:= \overline{\Omega_{w,h}^\circ}$.  
Our main result gives a combinatorial characterization of $\Omega_{w,h}$ for certain permutations $w$. 

\begin{Remark} 
One may define the Hessenberg Schubert cells and opposite Hessenberg Schubert cells in the language of Bia\l{}ynicki--Birula strata, as in \cite{Cho-Hong-Lee1}.  However, our definition is equivalent, and the Morse-theoretic point of view is not necessary for our purposes. 
\end{Remark}

\subsection{Subsets of Weyl type and acyclic orientations}

As mentioned above, one of the contributions of this note is to introduce the theory of subsets of Weyl type into the study of Hessenberg Schubert closure relations. We briefly recall the relevant terminology and results. 

Let $h: [n]\to [n]$ be a Hessenberg function.  Then $h$ determines a subset of $\Phi^+$ defined by
\[
\Phi_h^+ = \{ t_i-t_j \in \Phi \mid i<j \textup{ and } j\leq h(i)  \}
\]
and similarly, we let $\Phi_h^- = \{ t_j -t_i \mid i < j \, \textup{ and }\,  j\leq h(i) \} \subseteq \Phi^-$.
Note that $\ell_h(w) = | N(w)\cap \Phi_h^+ |$ where $\ell_h$ is the Hessenberg length function defined in~\eqref{eqn.Hess.length} above.

\begin{Definition}\label{definition: subset Weyl type} 
Given a subset $\cs \subseteq \Phi_h^+$ we say that $\cs$ is \textbf{$\Phi_h^+$-closed} if for all $\alpha, \beta\in \cs$ such that $\alpha+\beta \in \Phi_h^+$, then $\alpha+\beta\in \cs$ as well.  Given such a subset $\cs \subseteq \Phi_h^+$, we say that $\cs$ is \textbf{a subset of Weyl type (with respect to $h$)} if both $\cs$ and its complement $\Phi_h^+ \setminus \cs$ are $\Phi_h^+$-closed.  Denote the set of all subsets $\cs \subseteq \Phi_h^+$ of Weyl type (with respect to $h$) by $\cw_h$.
\end{Definition}

It is a well known theorem of Kostant~\cite[Prop.~5.10]{Kos61} that $\cs\subseteq \Phi^+$ is a subset of Weyl type if and only if $\cs = N(w)$ for some $w\in \Symm_n$.  Sommers and Tymoczko generalized that result to the setting of subsets of Weyl type in $\Phi_h^+$.  The following summarizes their results from~\cite{SomTym06} in the form most useful for our purposes.

\begin{Theorem}[Sommers--Tymoczko \cite{SomTym06}]\label{thm.Sommers-Tymoczko} Let $h: [n]\to[n]$ be a Hessenberg function and $\cs\in \cw_h$.  
\begin{enumerate}
\item  There exists $w\in \Symm_n$ such that $\cs = N(w)\cap \Phi_h^+$, and $\cs$ is a subset of Weyl type with respect to $h$ if and only if it is of this form.
\item There exists a unique element $z_\cs\in \Symm_n$ satisfying both $\cs = N(z_\cs)\cap \Phi_h^+$ and $z_\cs^{-1}(-\Delta)\cap \Phi^+ \subseteq \Phi_h^+$.
\item $N(z_\cs)\subseteq N(y)$ for any $y\in \Symm_n$ with $\cs \subseteq N(y)$.
\end{enumerate}
\end{Theorem}

Given a fixed $\cs\in \cw_h$, we now consider
\[
\cw(\cs, h) := \{ w\in \Symm_n \mid N(w)\cap\Phi_h^+ = \cs \} \subseteq \Symm_n
\]
i.e., $\cw(\cs,h)$ is the set of permutations whose associated subset of Weyl type is exactly $\cs$. Note that $\cw(\cs,h)$ is always non-empty for any $\cs \in \cw_h$ by Theorem~\ref{thm.Sommers-Tymoczko}, and we obtain a partition $\Symm_n = \bigsqcup_{\cs\in \cw_h} \cw(\cs,h)$.   Recall that (left) \textbf{weak Bruhat order} is the partial order on $\Symm_n$ defined by
\[
u \leq_L v  \;  \textup{ if } \;  v= s_{i_1}\cdots s_{i_k} u \textup{ for simple reflections $s_{i_1}, \ldots, s_{i_k}$ such that $\ell(v) = \ell(u)+k$.}
\]
The weak Bruhat order is stronger than Bruhat order in the sense that $u \leq_L v$ implies $u\leq v$ for all $u,v\in \Symm_n$.  Note that $u\leq_L v$ if and only if $N(u) \subseteq N(v)$ \cite[Prop.~3.1.3]{BjoBre05}.
The following lemma tells us that $\cw(\cs,h)$ is a weak Bruhat interval.  A proof can be found in Appendix~\ref{sec.appendix}. 

\begin{Lemma}\label{lemma.interval} 
Let $h: [n]\to[n]$ be a Hessenberg function and $\cs\in \cw_h$.  There exist elements $z_\cs, w_\cs \in \cw(\cs, h)$ such that $\cw(\cs,h)$ is precisely the weak (left) Bruhat interval 
\[
[z_\cs, w_\cs]_{L} := \{ v\in \Symm_n \mid z_\cs \leq_L v \leq_L w_\cs \}.
\]
\end{Lemma}

In order to apply the results of~\cite{Cho-Hong-Lee1} needed below, we now introduce a graph $\Gamma_h$ uniquely determined by a Hessenberg function $h$.   

\begin{Definition}  Let $h:[n] \to [n]$ be a Hessenberg function.  The \textbf{incomparability graph} $\Gamma_h = (V(\Gamma_h), E(\Gamma_h))$ is the graph on vertex set $V(\Gamma_h) = [n]$ with edges $E(\Gamma_h):= \{ \{i,j\} \mid j<i \textup{ and } i\leq h(j) \}$.
\end{Definition}

\begin{Example}\label{example: Gwh} 
The incomparability graph $\Gamma_h$ for $h=(2,4,4,4)$ and $h=(3,4,5,5,5)$ are given below.
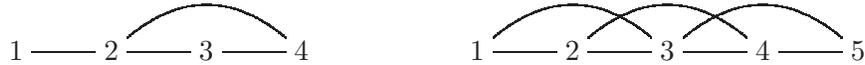
\vspace*{.15in}
\[\xymatrix{1 \ar@{-}[r] & 2 \ar@{-}[r] \ar@{-}@/^1.5pc/[rr] & 3 \ar@{-}[r] & 4 & & 
1 \ar@{-}[r]\ar@{-}@/^1.5pc/[rr]  & 2 \ar@{-}[r]\ar@{-}@/^1.5pc/[rr]  & 3 \ar@{-}[r] \ar@{-}@/^1.5pc/[rr]  & 4 \ar@{-}[r]  & 5
}\]
\end{Example}

The incomparability graph for $h$ plays a key role in the results of~\cite{Harada-Precup2019} and also appears in~\cite{Cho-Hong-Lee1, Cho-Hong-Lee2} (with the notation $G_{e,h}$).  An \textbf{acyclic orientation} $o$ of $\Gamma_h$ is an assignment of a direction (i.e.~orientation) to each edge $e\in E(\Gamma_h)$ such that the resulting oriented graph contains no directed cycles.  Given $\cs\in \cw_h$, we obtain an orientation $o_h(\cs)$ of $\Gamma_h$ defined by the rule
\begin{equation}\label{eq: def orientation}
 \xymatrix{ j & \ar[l] i } \; \textup{ for $\{i,j\}\in E(\Gamma_h)$ with $j<i$ if and only if $t_j-t_i \in \cs$.  } 
\end{equation} 
In other words, $o_h(\cs)$ is obtained by orienting edges corresponding to the roots in $\cs$ to the left, and oriented all other edges to the right.

\begin{Example}\label{ex.acyclic} Let $n=4$.  Consider the following acyclic orientations, the first of $\Gamma_{h}$ for $h=(3,4,4,4)$, and the second of $\Gamma_{h}$ for $h=(2,3,4,4)$.

\vspace*{.15in}
\[\xymatrix{1 \ar[r]  & 2 \ar@/^1.5pc/[rr] & 3  \ar@/_1.5pc/[ll] \ar[l]\ar[r]  & 4  
&&&
1 & 2 \ar[l] \ar[r]& 3 \ar[r] & 4
}\]

The acyclic orientation on the LHS is $o_{(3,4,4,4)}(\{t_2-t_3, t_1-t_3 \})$ and the one on the RHS is $o_{(2,3,4,4)}(\{ t_1-t_2 \})$.  
\end{Example}

The following observation yields a bijection between acyclic orientations and subsets of Weyl type. 

\begin{Lemma}\label{Lemma.orientations} The set of all acyclic orientations of $\Gamma_h$ is precisely $\{ o_h(\cs) \mid \cs\in \cw_h \}$.
\end{Lemma}

\begin{proof}[Sketch of Proof] 
Consider first the case in which $h=(n,n,\ldots, n)$ so $\Gamma_h=K_n$ is the complete graph on $n$ vertices.  It is a simple exercise to show that there are precisely $n!$ acyclic orientations of $K_n$, each uniquely determined by a permutation $w$ according to the rule that $j\longleftarrow i$ if and only if $j<i$ and $w(j)>w(i)$, or equivalently, $t_j -t_i\in N(w)$.
Since every subset of Weyl type of the form $N(w)$ for a unique $w\in \Symm_n$, the claim now follows.

Returning to the case of an arbitrary Hessenberg function, note that $\Gamma_h$ is a subgraph of $K_n$.  Let $\cs\in \cw_h$. By Theorem~\ref{thm.Sommers-Tymoczko} there exists $w\in \Symm_n$ such that $\cs=N(w)\cap \Phi_h^+$. The orientation $o_h(\cs)$ defined in~\eqref{eq: def orientation} is the orientation induced by the that of $w$ on $K_n$ as in the previous paragraph.  This shows that $o_h(\cs)$ is an acyclic orientation.  Since every acyclic orientation of $\Gamma_h$ is the restriction of an acyclic orientation on $K_n$ it follows that every acyclic orientation is of this form.
\end{proof} 

\begin{Remark} 
In~\cite{Cho-Hong-Lee2}, the authors define and study an equivalence class of permutations. 
In their notation, the equivalence class $[w]_h$ from~\cite[Definition 3.4]{Cho-Hong-Lee2} is precisely the set $\cw(\cs,h)$ above for $\cs = N(w)\cap \Phi_h^+$. 
We also note that in \cite{Cho-Hong-Lee1, Cho-Hong-Lee2} the authors use the language of subgraphs $G_{w,h}$ of $\Gamma_h$ for 
different permutations $w$, but this data can be equivalently characterized by acyclic orientations, which is what we choose to do in this manuscript. 
\end{Remark}

The relation between $\cw(\cs,h)$ (and specifically the maximal element $w_\cs$ of $\cw(\cs,h)$) to the acyclic orientation $o_h(\cs)$ is developed further in the next section.

\section{Reachability}\label{sec.reachability}

A key contribution of this manuscript is to connect the work of Cho, Hong, and Lee in \cite{Cho-Hong-Lee1}---in which they use the combinatorial notion of reachability to study Hessenberg Schubert cells---to the theory of subsets of Weyl type. In this section, we make this connection precise. The essential result is Proposition~\ref{prop.reachability}, which is the technical engine driving the proof of our main theorem (Theorem~\ref{theorem: main}). 

We begin with the definition of reachability, taken from \cite{Cho-Hong-Lee1}, 
relating two vertices $i$ and $j$ on the incompatibility graph $\Gamma_h$.

\begin{Definition}  
Let $h: [n] \to [n]$ be a Hessenberg function and $\Gamma_h$ its associated incomparability graph, equipped with an 
acyclic orientation $o_h(\cs)$. Suppose $i>j$.  We say that $i$ is \textbf{reachable from $j$ with respect to $\cs$
 (or $o_h(\cs)$)} if $i=v_m$ and $j=v_0$ and there exists a sequence of vertices $j=v_0< v_1< \cdots < i=v_m$ 
 of $\Gamma_h$
 such that  there is an oriented edge from each $v_\ell$ to $v_{\ell+1}$, i.e., 
 there is a sequence of oriented edges of the form 
 $j=v_0 \longrightarrow v_1 \longrightarrow \cdots \longrightarrow v_{m-1}\longrightarrow v_m = i$ in $\Gamma_h$ (equipped 
 with the orientation $o_h(\cs)$).  We allow $m$ 
 to be $0$, that is, $j$ is always reachable from $j$. 
\end{Definition}

We say that a vertex $k$ in an oriented graph is a \textbf{source} (with respect to the given orientation)
 if all edges adjacent to $k$ point ``out'' of $k$, i.e.,
each such edge is of the form $k \longrightarrow i$ for all $i$ adjacent to $k$.  Since our incomparability graphs have vertices labelled by sets of 
positive integers $\{1,2,\cdots, n\}$, we say that a vertex $k$ is the \textbf{largest source} if $k$ is a source and moreover, 
if $j$ is another source, then $k>j$. The next lemma shows that any vertex larger than the largest source $k$ is reachable 
from $k$. 

\begin{Lemma} \label{lemma.largest.source} Let $h:[n]\to [n]$ be a Hessenberg function, $\Gamma_h$ its 
associated incomparability graph, and $\cs\in \cw_h$. 
Suppose that $k$ is the largest source of $\Gamma_h$ with respect to the acyclic orientation $o_h(\cs)$ and $i$ is a vertex
with $i>k$. Then $i$ is reachable from $k$. 
\end{Lemma}

\begin{proof}

To prove the claim of the lemma, it suffices to show that the set
\begin{eqnarray}\label{eqn.no.source}
\{ i>k \mid \textup{ $i$ is not reachable from $k$ }  \}
\end{eqnarray}
is empty.  For the sake of obtaining a contradiction, suppose not, and consider the oriented subgraph $\Gamma'$ induced by the vertices in set~\eqref{eqn.no.source}.  Since the original oriented graph $\Gamma_h$ is 
acyclic, so is the subgraph. Any acyclic orientation must have a source, 
 so there exists a vertex $i'$ which is a source of $\Gamma'$. 
We claim that the the vertex $i'$ must also be a source of the original oriented graph $\Gamma_h$. 
To see this, we must show that any edge $\{j, i'\}\in E(\Gamma_h)$ has orientation $i'\longrightarrow j$.
To argue this, 
we first observe that $i' > h(k)$, because 
if $i' \leq h(k)$ then by definition of $\Gamma_h$ there would be an edge from $k$ to $i'$, and since $k$ is a source
by assumption, the edge between $k$ and $i'$ would be oriented as $k\longrightarrow i'$, making $i'$ reachable from $k$. 
This contradicts the assumption that $i'$ is in the set~\eqref{eqn.no.source}. 
This implies that there is no edge in $\Gamma_h$ between $i'$ and any vertex $j$ with $j \leq k$. 
Thus we may now assume without loss of generality that $j > k$. We take cases.

Suppose that $k < j < i'$. If there is no edge between $j$ and $i'$ in $\Gamma_h$ then there is nothing to prove, so 
suppose there is an edge. If $j$ is not reachable by $k$, then $j$ is in the set~\eqref{eqn.no.source} and we have 
assumed that $i'$ is a source of the subgraph, so we must have $i' \longrightarrow j$. On the other hand, if $j$ is reachable
by $k$, then by reasoning similar to the above, we must also have $j \longleftarrow i'$ since otherwise $i'$ would be reachable by
$k$. 

Next suppose  $j>i'$. Again, if there is no edge between $j$ and $i'$ then there is 
nothing to prove, so we may suppose $\{i',j\}$ is an edge of $\Gamma_h$.  If $j$ is not reachable by $k$ then 
by the same reasoning as above we already know the edge is oriented as $i' \longrightarrow j$, so suppose $j$ is reachable by $k$.
Recall that we wish to show $i' \longrightarrow j$ in $o_h(\cs)$. Suppose for the sake of contradiction that $i'\longleftarrow j$ instead.  
Since $j$ is reachable from $k$, there exists a sequence $v_0 < v_1< \cdots < v_m$ of vertices such that $k=v_0 \longrightarrow j_1 \longrightarrow \cdots \longrightarrow v_{m-1} \longrightarrow v_m=j$ in $o_h(\cs)$.  Let $v_\ell$ be the largest vertex in that list such that $v_\ell<i'$.  Since $\{v_\ell, v_{\ell+1}\}$ is an edge in $\Gamma_h$ and $i'<v_{\ell+l}$ we conclude that $\{v_\ell, i'\}\in E(\Gamma_h)$.  Since $v_\ell$ is reachable from $k$, by the same reasoning as above we know that $v_\ell \longleftarrow i'$ and we may visualize the graph as 
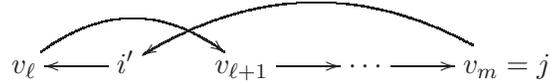
\vspace*{.15in}
\[
\xymatrix{ v_\ell \ar@/^1.5pc/[rr] & \ar[l]  i' &  v_{\ell+1} \ar[r] & \cdots \ar[r] &v_m= j \ar@/_2pc/[lll]  }
\]  
in $o_h(\cs)$, which shows that there is a cycle starting and ending at $i'$, contradicting the fact that $o_h(\cs)$ is an acyclic orientation. 
Thus we must have $i' \longrightarrow j$. 

We have now shown that $i'$ is a source in $\Gamma_h$, contradicting that $k$ is the largest source. 
Therefore~\eqref{eqn.no.source} is indeed empty as was to be shown.
\end{proof}

The next lemma shows that reachability implies an inequality among entries in the one-line notation of $w$ for $w \in \cw(\cs,h)$.

\begin{Lemma}\label{lemma: reachable and w}
Let $j\leq i$ be vertices in $\Gamma_h$ and $\cs\in \cw_h$.  If $i$ is reachable from $j$ with respect to $o_h(\cs)$, 
then $w(j)\leq w(i)$ for all $w\in \cw(\cs,h)$.
\end{Lemma}

\begin{proof} Let $w\in \cw(\cs,h)$.  Suppose $i$ is reachable from $j$ so we have a sequence of vertices $v_0< v_1< \cdots< v_m$ such that $j=v_0 \longrightarrow v_1 \longrightarrow \cdots \longrightarrow v_{m-1} \longrightarrow v_m =i$ in $o_h(\cs)$.  Thus for each $\ell$ such that $1\leq \ell \leq m$, we have $t_{v_{\ell-1}} - t_{v_{\ell}} \in \Phi_h^+\setminus\cs$ and since $\cs = N(w)\cap \Phi_h^+$, 
we get $t_{v_{\ell-1}} - t_{v_{\ell}} \notin N(w)$. This means $w(v_{\ell-1}) \leq w(v_\ell)$ for all  $1 \leq \ell \leq m$, and by putting the inequalities together we obtain $w(j)\leq w(i)$, as desired. 
\end{proof}

It also turns out that the location of $1$ in the one-line notation of $w \in \cw(\cs,h)$ is significant.

\begin{Lemma} \label{lemma.source1} Let $h:[n]\to [n]$ be a Hessenberg function and $\cs\in \cw_h$. If $w\in \cw(\cs,h)$, then 
$w^{-1}(1)$ is a source of $\Gamma_h$ equipped with the orientation $o_h(\cs)$.
\end{Lemma}

\begin{proof}  Suppose $j=w^{-1}(1)$.  It follows directly from the definition of the inversion set that
\[
\{ t_1-t_j, t_2-t_j, \ldots, t_{j-1}-t_j \} \subseteq N(w)
\] 
since $w(j)=1$ is strictly smaller than any $w(1), w(2), \cdots, w(j-1)$. By similar reasoning, since $w(j)=1$ is also smaller than 
$w(j+1), \cdots, w(n)$, we have 
\[
\{t_j - t_{j+1}, t_j -t_{j+2}, \ldots, t_j-t_n\} \subseteq \Phi^+ \setminus N(w).
\]
As $w \in \cw(\cs,h)$ by hypothesis, we have that $\cs = N(w)\cap \Phi_h^+$. 
From the definition of $o_h(\cs)$ in~\eqref{eq: def orientation} we see that the edges $\{k,j\}$ with $k<j$ must be 
oriented to the left and the edges for $k>j$ are oriented toward the right. 
Thus $j$ is a source, as desired. 
\end{proof}

We now give an inductive construction using the Hessenberg function $h$ and graph $\Gamma_h$ that will be useful in what follows.  
Let $k\in [n]$. Consider the smaller graph on $n-1$ vertices which is obtained from $\Gamma_h$ by deleting vertex $k$ and all adjacent edges to $k$, and for convenience in our arguments below, relabeling the vertex set to be $\{2,3,\ldots,n\}$ (so $\{1,2,\ldots, k-1\}$ gets relabelled as $\{2,3,\ldots,k\}$ respectively, and the labels of $\{k+1,\ldots,n\}$ are unchanged). We let $h^{(k)}: \{2,3,\ldots,n\} \to \{2,3,\ldots,n\}$ denote the Hessenberg function whose associated graph $\Gamma_{h^{(k)}}$ is precisely the graph just described. 
Alternatively, suppose the Hessenberg function $h$ is visualized as a collection of boxes in an $n \times n$ array where the $(i,j)$-th box (in row $i$ and column $j$) is said to be in the collection if $i\leq h(j)$.
Then the collection of boxes in the $(n-1) \times (n-1)$ array corresponding to $h^{(k)}$ is obtained from that of $h$ by deleting the $k$-th row and column.

Set $\Psi:= \{t_i -t_j \mid i\neq j \textup{ and } i,j \in \{2,\ldots, n\}\}$.  Then $\Psi$ is a root system of type $A_{n-2}$, and the only difference between $\Psi$ and the standard root system is that we have shifted the index set to be $\{2,3,\cdots,n\}$ instead of $\{1,2,\cdots,n-1\}$.
Fix $k \in [n]$. In analogy with how we defined $\Phi_h^+$ above, let us define $\Psi^+_{h^{(k)}} := \{ t_i -t_j \in \Psi \mid i<j \textup{ and } j\leq h^{(k)}(i) \}$. 
From the description of the graph $\Gamma_{h^{(k)}}$ it is not hard to check that 
\begin{equation}\label{equation: induction} 
\Psi^+_{h^{(k)}} = u_k (\Phi_h^+) \cap\Psi
\end{equation} 
where $u_k:= s_1s_2 \cdots s_{k-1} \in \Symm_n$ (in one-line notation we have $u_k = [2, 3, \ldots, k, 1, k+1, k+2, \ldots,n]$). Here we take $u_1 := e \in \Symm_n$. 
Given $\cs \in \cw_h$, we can consider the orientation or $\Gamma_{h^{(k)}}$ induced by $o_h(\cs)$, which must necessarily be acyclic. 
Let $\cs' \subseteq \Psi_{h^{(k)}}^+$ denote the corresponding subset of Weyl type. By~\eqref{equation: induction} we have 
$$\cs' = u_k(\cs)\cap \Psi.$$ 

The next lemma relates certain elements in $\cw(\cs, h)$ with those in $\cw(\cs', h^{(k)})$. 

\begin{Lemma}\label{lemma: cs and cs prime} 
Let $y \in \langle s_2, s_3, \cdots, s_{n-1} \rangle$ and  $k$ be a source of $\Gamma_h$ with respect to $o_h(\cs)$ for $\cs \in \cw_h$. Let $u_k$ be as above and $w = yu_k\in \Symm_n$. Then $y \in \cw(\cs', h^{(k)})$ if and only if $w \in \cw(\cs, h)$. 
\end{Lemma} 

\begin{proof} 
First suppose $w \in \cw(\cs, h)$, so $N(w) \cap \Phi_h^+ = \cs$. On the other hand, since the decomposition $w = yu_k$ satisfies $\ell(w) = \ell(y) + \ell(u_k)$ (note $u_k$ is a minimal coset representative of $\Symm_{n-1} \backslash \Symm_n$) we have $N(w) = N(u_k) \sqcup u_k^{-1} N(y)$ (cf. for instance \cite[Section 1.7]{Humphreys}). Combining these facts we obtain 
\begin{equation}\label{eq: 1} 
\cs = \left(N(u_k) \sqcup u_k^{-1} N(y)\right) \cap \Phi_h^+  \, \, \Leftrightarrow  \, \, u_k(\cs) = \left(u_k N(u_k) \sqcup N(y)\right) \cap u_k (\Phi_h^+)
\end{equation} 
and thus 
$$
\cs' = u_k(\cs) \cap \Psi = N(y) \cap \Psi \cap u_k \Phi_h^+ = N(y) \cap \Psi^+_{h^{(k)}}
$$
where the first and third equalities follows from~\eqref{equation: induction} and the second equality follows from an explicit computation of $u_k N(u_k)$ which shows that $u_k N(u_k) \cap \Psi = \emptyset$. This proves $y \in \cw(\cs', h^{(k)})$. 

To see the other direction, suppose $y \in \cw(\cs', h^{(k)})$. We want to show $w \in \cw(\cs, h)$, for which it suffices to prove~\eqref{eq: 1}. Since $k$ is a source we have 
\[
\cs = \{ t_j-t_i \in \cs \mid i, j \in [n] \setminus \{k\}  \} \sqcup\{ t_j -t_k \mid j<k,\, k\leq h(j) \}.
\]

Since $u_k([n]\setminus \{k\}) = \{2,\ldots, n\}$ we have 
\begin{eqnarray}\label{eqn.Weyl.type.induction}
 \{ t_j-t_i \in \cs \mid i, j \in [n] \setminus \{k\}  \} = u_k^{-1}(\cs')
\end{eqnarray}
and now the claim follows from the observation that~\eqref{equation: induction} implies 
$$
u_k^{-1}(\cs') = u_k^{-1}\left(N(y) \cap \Psi^+_{h^{(k)}}\right) = 
u_k^{-1} N(y) \cap \Phi_h^+
$$ 
and that $N(u_k) \cap \Phi_h^+ = \{t_j - t_k \mid j<k ,\, k \leq h(j) \}$. This completes the proof. 
\end{proof}

The next lemma is a kind of converse to Lemma~\ref{lemma.source1}.

\begin{Lemma}\label{lemma.source2} Let $\cs\in \cw_h$. If $k$ is a source of $\Gamma_h$ with respect to the orientation 
$o_h(\cs)$, then there exists $w\in \cw(\cs,h)$ such that $w^{-1}(1)=k$.
\end{Lemma}

\begin{proof}
Since $\cw(\cs', h^{(k)})$ is non-empty, there exists 
$y\in \cw(\cs', h^{(k)}) \subseteq \left< s_2, \ldots, s_{n-1} \right>$.  Consider $w=yu_k$.  
By Lemma~\ref{lemma: cs and cs prime}, $w \in \cw(\cs, h)$, and by construction $w(k) =1$.  
\end{proof}

 The relationship is even tighter between the largest source and the maximal element $w_\cs$ of $\cw(\cs,h)$. 

\begin{Lemma} \label{cor.source} 
Let $h:[n] \to [n]$ be a Hessenberg function and $\cs\in \cw_h$.  Then the vertex $k$ is the largest source of $\Gamma_h$ with respect to $o_h(\cs)$ if and only if $w_\cs^{-1}(1)=k$.
\end{Lemma}

\begin{proof} Suppose $k$ is the largest source of $\Gamma_h$ with respect to $o_h(\cs)$.  By Lemma~\ref{lemma.source2} there exists $w\in \cw(\cs,h)$ such that $w(1)=k$ so we may write $w = yu_k$ for some $y\in \left< s_2, \ldots, s_{n-1}\right>$.  
Let $j :=w_\cs^{-1}(1)$ or equivalently, $w_\cs(j)=1$. Then by Lemma~\ref{lemma.source1} we know $j$ must be a source, 
and since $k$ is the largest source, we conclude $j \leq k$. To complete the argument that $j=w_\cs^{-1}(1) = k$ we show that $k \leq j$. 
To see this, write $w_\cs = y_\cs u_j$ for some $y_\cs\in \left< s_2, \ldots,s_n \right>$, which is possible since $j=w_\cs^{-1}(1)$.  
We know that $w_\cs$ is the maximal element of $\cw(\cs,h)$, so 
\[
w\leq w_\cs \Rightarrow y u_k \leq y_\cs u_j \Rightarrow u_k\leq u_j \Rightarrow k\leq j
\]
where the second implication follows from the fact that the map from $\Symm_n$ to the shortest coset representatives $\{u_1 :=e, u_2, \cdots, u_n\}$ of the subgroup $\langle s_2, s_3, \cdots, s_n \rangle$ is order-preserving \cite[Prop. 2.5.1]{BjoBre05}. 
Thus $j=w_\cs^{-1}(1)=k$ as desired. The converse follows by similar reasoning.
\end{proof}

The lemma and other inductive structure established above yield a simple way to compute $w_\cs$ given the corresponding acyclic orientation. 

\begin{Example}\label{ex.interval} Let $n=4$.  We consider as in Example~\ref{ex.acyclic} the orientation $o_{(3,4,4,4)}(\{t_2-t_3, t_1-t_3 \})$ of graph $\Gamma_{h}$ for $h=(3,4,4,4)$.  Using the graph pictured in that example, Lemma~\ref{cor.source} tells us that $w_\cs(3)=1$.  Deleting $3$ and all adjacent edges tells us that $w_\cs(1)=2$.  Continuing in this way gives the sequence of graphs:
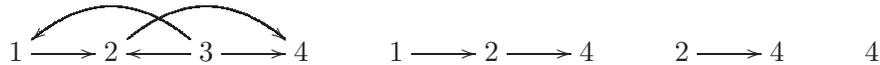
\vspace*{.15in}
\[\xymatrix{1 \ar[r]  & 2 \ar@/^1.5pc/[rr] & 3  \ar@/_1.5pc/[ll] \ar[l]\ar[r]  & 4  
&
1 \ar[r]  & 2 \ar[r] & 4 
&
2\ar[r]&4
&
4
}\]
and $w_\cs = [2,3,1,4].$
\end{Example}

We can now prove one of our important technical results, which characterizes reachability in terms of the maximal elements $w_\cs$ for $\cs \in \cw_h$.  This reformulation is the tool which allows us to prove our characterization of $T$-fixed points in the opposite Hessenberg Schubert variety in the next section. 

\begin{Proposition}\label{prop.reachability} Let $h: [n] \to [n]$ be a Hessenberg function and $\cs\in \cw_h$.  Suppose $j\leq i$.  Then $i$ is reachable from $j$ with respect to $\cs$ if and only if $w_\cs(j) \leq w_\cs(i)$. 
\end{Proposition}

\begin{proof} 
First suppose that $j \leq i$ and $i$ is reachable from $j$. We wish to show that $w_\cs(j) \leq w_\cs(i)$. This follows immediately from Lemma~\ref{lemma: reachable and w}. 

So now suppose $j \leq i$ and that $w_\cs(j) \leq w_\cs(i)$. We need to show that $i$ is reachable from $j$. We proceed to prove the contrapositive statement by an induction on $n$. When $n=1$, we have $j=i=1$ and $w_\cs=e$ and the claim is trivial.
Now suppose $n\geq 2$ and that the claim is true for $n-1$. Note that we may assume $j \neq i$ since if $j=i$ then $i$ is reachable from $j$ by convention and the claim is immediate. So suppose $j < i$ and additionally suppose that $i$ is not reachable from $j$.
Let $k :=w_\cs^{-1}(1)$ and write $w_\cs = y_\cs u_k$ for some $y_\cs \in \left<s_2, \ldots, s_{n-1} \right>$. If $i=k$, 
then our claim follows immediately since $w_\cs(k)=w_\cs(i) = 1<w_\cs(j)$. 
Next, notice that by Corollary~\ref{cor.source}, $k$ is the largest source of $\Gamma_h$ with respect to $o_h(\cs)$. 
If $j=k$, then by Lemma~\ref{lemma.largest.source} we would have that $i$ is reachable from $j$, but we have assumed that $i$ is not reachable from $j$, so we conclude $j \neq k$. 

We now have that $i \neq k, j \neq k$ and $j < i$. 
Consider the graph $\Gamma_{h^{(k)}}$ obtained from $\Gamma_h$ with acyclic orientation $o_{h^{(k)}}(\cs')$ induced by $o_h(\cs)$. 
By Lemma~\ref{lemma: cs and cs prime}, since $w_\cs \in \cw(\cs, h)$ we know $y_\cs \in \cw(\cs', h^{(k)})$. We wish to show $y_\cs$ is the maximal element in $\cw(\cs', h^{(k)})$. To see this, suppose $z_\cs$ is the maximal element. First observe $y_\cs \leq z_\cs$ since $z_\cs$ is maximal. On the other hand, $z_\cs u_k \in \cw(\cs, h)$ by Lemma~\ref{lemma: cs and cs prime} and since $w_\cs$ is maximal, we conclude $z_\cs u_k \leq w_\cs = y_\cs u_k$. Now the properties of Bruhat order imply $z_\cs \leq y_\cs$. Hence $y_\cs = z_\cs$ and $y_\cs$ is maximal in $\cw(\cs', h^{(k)})$.  
Our assumptions imply that $i'=u_k(i)$ is not reachable from $j'=u_k(j)$.  As $\Gamma_{h^{(k)}}$ is a Hessenberg graph on $n-1$ vertices and $y_\cs$ is maximal, we know $y_\cs(i')<y_\cs(j')$ by induction. Since $w_\cs = y_\cs u_k$, we conclude $w_\cs(i)<w_\cs (j)$ as desired.
\end{proof}

As a first geometric application of the results in this section, we prove the following. Recall that if $v \in \cw(\cs,h)$ then $v \leq_L w_\cs$ so 
there exists (by definition of weak Bruhat order) an element $u \in \Symm_n$ such that $w_\cs = u^{-1}v$ with $\ell(w_\cs) = \ell(v)+\ell(u)$. 

\begin{Proposition}\label{prop.cell-translation} 
Suppose $\cs\in \cw_h$ and $v\in \cw(\cs,h)$ and let $w_\cs = u^{-1}v$ for $u\in \Symm_n$ 
where $\ell(w_\cs) = \ell(v)+\ell(u)$. Then
\[
\Omega_{v,h}^\circ = u \left( \Omega_{w_\cs}^\circ \cap \Hess(u^{-1}\mathsf{S}u, h)\right).
\]
\end{Proposition}
\begin{proof} Let $b$ denote a lower triangular matrix with $1$'s on the diagonal.  Recall that $wbB\in \Omega_w^\circ$ if and only if $b_{ij}=0$ for all $i>j$ such that $w(i)<w(j)$, or equivalently, $b_{ij}=0$ for all $t_j-t_i \in N(w)$. Our assumptions imply that $v\leq_L w_\cs$ so $N(v) \subseteq N(w_\cs)$.

Suppose $w_\cs bB\in \Omega_{w_\cs}^\circ$, so $b_{ij}=0$ for all $t_j-t_i\in N(w_\cs)$ implying that $b_{ij}=0$ for all $t_j-t_i\in N(v)$ also.  Thus $uw_\cs bB = vbB\in \Omega_v^\circ$.  We now have
\begin{eqnarray*}
u\Omega_{w_\cs}^\circ \subseteq \Omega_v^\circ &\Rightarrow&  u\Omega_{w_\cs}^\circ \cap \Hess(\mathsf{S},h) \subseteq \Omega_{v}^\circ\cap \Hess(\mathsf{S},h)\\
&\Rightarrow& u \left( \Omega_{w_\cs}^\circ \cap \Hess(u^{-1}\mathsf{S}u, h)\right) \subseteq \Omega_{v,h}^\circ.
\end{eqnarray*}
Here the second implication uses the fact that $\Hess(\mathsf{S}, h) = u\Hess(u^{-1}\mathsf{S}u, h)$.

To prove the opposite inclusion, suppose $vbB\in \Omega_{v,h}^\circ$. By \cite[Corollary 3.9]{Cho-Hong-Lee1}, $b_{ij}=0$ for all $i>j$ such that $i$ is not reachable from $j$.  By Proposition~\ref{prop.reachability}, this implies $b_{ij}=0$ for all $i>j$ such that $w_\cs(i)<w_\cs(j)$ so $w_\cs bB \in \Omega_{w_\cs}^\circ$.
Furthermore, since $vbB\in \Hess(\mathsf{S},h)$ it follows immediately that $u^{-1}vbB = w_\cs bB\in \Hess(u^{-1}\mathsf{S}u,h)$.  Thus $u^{-1}v b B \in \Omega_{w_\cs}^\circ \cap \Hess(u^{-1}\mathsf{S}u, h)$, as desired.
\end{proof}

\section{Connection to Bruhat Order}

In this section we state and prove our main result, Theorem~\ref{theorem: main}.
To do so, we need some terminology and notation from the work of Cho, Hong, and Lee \cite{Cho-Hong-Lee1}.

Let $n$ be a positive integer and $1 \leq k \leq n$.  Denote by $I_{k,n}$ the set of $k$-tuples of positive integers $\underline{i} = (i_1, i_2, \cdots, i_k) \in \Z^k$ satisfying $1 \leq i_1 < i_2 < \cdots < i_k \leq n$. Suppose we are given a permutation $w \in S_n$ and $\underline{i} \in I_{k,n}$. We can consider the $k$-tuple of not-necessarily-strictly-increasing integers $(w(i_1), w(i_2), \cdots, w(i_k)) \in \Z^k$ and then re-order the entries in such a way that they are strictly increasing;  we denote the result as $w \cdot \underline{i}$, and by construction, we have $w \cdot \underline{i} \in I_{k,n}$.

We also need a certain subset $J_{w,h,k}$ of $I_{k,n}$ associated to a permutation $w \in S_n$, a Hessenberg function $h: [n] \to [n]$, and an integer $k$ such that $1 \leq k \leq n$.   
To define it, we need the following terminology from \cite{Cho-Hong-Lee1}, which extends the notion of reachability (described in Section~\ref{sec.reachability}) to two subsets of $[n]$.

\begin{Definition} 
Let $A = \{1 \leq a_1 < a_2 < \cdots < a_r \leq n\}$ and $B = \{1 \leq b_1 < b_2 < \cdots < b_r \leq n\}$ be two subsets of $[n]$ of the same cardinality $r$. We say $A$ is \textbf{reachable} from $B$ if there exists a permutation $\sigma \in S_r$ such that $a_{\sigma(i)}$ is reachable from $b_i$ for all $1 \leq i \leq r$. 
\end{Definition} 

We now define $J_{w,h,k}$ as in  \cite[Equation (3.2)]{Cho-Hong-Lee1} in terms of reachability with respect to the acyclic orientation $o_h(\cs)$ determined by $\cs=N(w)\cap \Phi_h^+$: 
\begin{equation}\label{eq: J w h j} 
J_{w,h,k} := \{ \underline{i} = (i_1, i_2, \cdots, i_k) \in I_{k,n}  \mid  \{i_1, i_2 \cdots, i_k\} \,  \textup{ is reachable from } \,  \{1,2,\cdots, k\} \}.
\end{equation}

With this notation in place, we can state the following result of Cho, Hong, and Lee \cite[Theorem 3.5]{Cho-Hong-Lee1}.

\begin{Theorem}[Cho--Hong--Lee] \label{thm.CHL}
Let $h: [n] \to [n]$ be a Hessenberg function and $w\in \Symm_n$.  Then 
\[
(\Omega_{w,h})^T = \{ u\in \Symm_n \mid u\cdot (1, \ldots, k) \in \{ w\cdot(i_1, \ldots, i_k) \mid (i_1, \ldots, i_k) \in J_{w,k,h}   \} \textup{ for all } 1\leq k \leq  n-1  \}.
\]
\end{Theorem}

Here and below we define and use the following partial order $\leq$ on $I_{k,n}$: 
\begin{equation}\label{eq: partial order on Ikn}
(i_1, i_2, \cdots, i_k) \geq (j_1, j_2, \cdots, j_k) \, \textup{ if and only if } \,  i_\ell \geq j_\ell \, \textup{ for all} \, \ell,\,  1 \leq \ell \leq k. 
\end{equation}
In order to obtain a description of $(\Omega_{w,h})^T$ in terms of Bruhat order using Theorem~\ref{thm.CHL}, we need the following lemma, which characterizes the set $J_{w,h,k}$ defined in ~\eqref{eq: J w h j} in terms of the maximal element $w_\cs$ of $\cw(\cs,h)$. This connects the discussion to that of Section~\ref{sec.reachability} and allows us to use the results therein.

\begin{Lemma}\label{lemma.reachability} Suppose $\cs\in \cw_h$.  For all $w\in \cw(\cs,h)$ and all  $k$ such that $1\leq k \leq n-1$, we have 
\[
J_{w,h,k} = \{ (i_1, \ldots, i_k) \in I_{k,n} \mid w_\cs \cdot (i_1,\ldots, i_k) \geq w_\cs \cdot (1,\ldots,k) \}
\]
where $J_{w,h,k}$ is the set defined by~\eqref{eq: J w h j}. 
\end{Lemma}

\begin{proof}
By its definition in~\eqref{eq: J w h j}, and from the definition of reachability, $J_{w,h,k}$ consists of $(i_1, \cdots, i_k) \in I_{k,h}$ such that there exists a permutation $\sigma \in \Symm_k$ with the property that $i_{\sigma(\ell)}$ is reachable from $\ell$ for all $1 \leq \ell \leq k$ with respect to $o_h(\cs)$. On the other hand, by Proposition~\ref{prop.reachability} this is equivalent to $w_\cs(i_{\sigma(\ell)}) \geq w_\cs(\ell)$ for all $1 \leq \ell \leq k$.  The claim now follows from our definition of the notation $w_\cs \cdot(i_1, \cdots, i_k) \geq w_\cs \cdot (1,\ldots, k)$.  
\end{proof}

We also need the following description of the Bruhat order in $\Symm_n$, which we now briefly recall (cf.~\cite[Propositions 2.4.8, 2.5.1, and Theorem 2.6.3]{BjoBre05}). 

\begin{Lemma} \label{lemma: Bruhat}
  Let $w,v\in \Symm_n$.
  Then $w \leq v$ in Bruhat order if and only
  if 
\begin{equation}\label{eq:tableau}
w\cdot (1, \ldots, k)\leq v\cdot (1, \ldots, k) \textup{ for all } 1\leq k \leq n-1.
\end{equation}
\end{Lemma}

We are now ready to state and prove the main result of this note. 

\begin{Theorem}\label{theorem: main}
 Let $h: [n]\to [n]$ be a Hessenberg function and fix $\cs\in \cw_h$.  Then 
\begin{equation}\label{eq: main} 
\Omega_{w_\cs, h}^T = [w_\cs,w_0].
\end{equation} 
where $[w_\cs, w_0]$ denotes the Bruhat interval of $w\in \Symm_n$ such that $w_\cs \leq w$.
\end{Theorem}
\begin{proof} First, as $\Omega_{w_\cs,h}^\circ \subseteq \Omega_w\cap \Hess(\mathsf{S},h)$ and the intersection $\Omega_w\cap \Hess(\mathsf{S},h)$ is closed, we have $\Omega_{w,h} \subseteq \Omega_w\cap \Hess(\mathsf{S},h)$.  It follows immediately that $\Omega_{w_\cs, h}^T \subseteq [w_\cs,w_0]$.

To prove the opposite inclusion, suppose $u \in [w_\cs, w_0]$.  By Lemma~\ref{lemma: Bruhat} we obtain
\begin{eqnarray}\label{eqn1}
w_\cs\cdot (1, \ldots, k) \leq u \cdot (1, \ldots, k) \; \textup{ for all } \; 1\leq k \leq n-1.
\end{eqnarray}
As $w_\cs$ is a permutation there exists $i_1, \ldots, i_{n-1} \in [n]$ such that $w_\cs (i_\ell) = u(\ell)$ for all $1\leq \ell \leq n-1$.  Equation~\eqref{eqn1} now yields
\[
w_\cs\cdot (i_1, \ldots, i_k) = u \cdot (1, \ldots, k) \geq w_\cs\cdot (1, \ldots, k)  \; \textup{ for all } \; 1\leq k \leq n-1.
\]
Now $ (i_1, \ldots, i_k) \in J_{w_\cs,h,k}$ by Lemma~\ref{lemma.reachability} and therefore $u\in \Omega_{w_\cs,h}^T$ by Theorem~\ref{thm.CHL}.
\end{proof}

Let $w_0=[n,n-1, \ldots, 2, 1]$ denote the longest element in $\Symm_n$.
By translating by $w_0$, we easily obtain the analogous result for Hessenberg Schubert cells corresponding to the minimal element $z_\cs$ of $\cw(\cs,h)$.

\begin{Corollary} Let $h: [n] \to [n]$ be a Hessenberg function and $\cs\in \cw_h$.  Then 
\[
X_{z_\cs ,h }^T = [e, z_\cs].
\]
where $[e, z_\cs]$ denotes the Bruhat interval of $w\in \Symm_n$ such that $w \leq z_\cs$.
\end{Corollary}
\begin{proof} Let $\bar{\cs} = \Phi_h^+ \setminus \cs\in \cw_h$.  Since $w_0B_-w_0^{-1} = B$ we have $w_0\Omega_{w}^\circ = X_{w_0w}^\circ$.  Using the fact that $w_0w_{\bar{\cs}} = z_{\cs}$ as in the proof of Corollary~\ref{cor.max} we obtain
\begin{eqnarray*}
w_0 \left(\Omega_{w_{\bar\cs } }^\circ\cap \Hess(w_0\mathsf{S}w_0, h) \right) = X_{z_\cs}^\circ \cap \Hess(\mathsf{S},h) = X_{z_\cs,h}^\circ.
\end{eqnarray*}
Since left multiplication respects closure relations in the flag variety, the corollary now follows by application of Theorem~\ref{theorem: main} to the Hessenberg Schubert variety $\overline{\Omega_{w_{\bar\cs } }^\circ\cap \Hess(w_0\mathsf{S}w_0, h)}$ and the fact that $w_{\bar\cs} \leq w\Leftrightarrow z_{\cs} \geq w_0w$ since left multiplication by $w_0$ is an order reversing involution of~$\Symm_n$.
\end{proof}

We can also give a reformulation of Theorem~\ref{thm.CHL} in the language of Bruhat order.

\begin{Corollary} \label{cor.reformulation}  Let $h: [n] \to [n]$ be a Hessenberg function and $\cs\in \cw_h$.  Suppose $v\in \cw(\cs,h)$ and let $w_\cs = u^{-1}v$ for $u\in \Symm_n$ such that $\ell(w_\cs) = \ell(v)+\ell(u)$.  Then 
\[
\Omega_{v,h}^T = u[w_\cs, w_0].
\]
\end{Corollary}
\begin{proof} 
Left multiplication respects closure relations in the flag variety, so by Proposition~\ref{prop.cell-translation} we have $\Omega_{v,h} = u\left( \overline{\Omega_{w_\cs}^\circ \cap \Hess(u^{-1}\mathsf{S}u, h)} \right)$.  Theorem~\ref{theorem: main} now yields 
\[
\left( \overline{\Omega_{w_\cs}^\circ \cap \Hess(u^{-1}\mathsf{S}u, h)} \right)^T = [w_\cs,w_0]
\] 
and the corollary follows.
\end{proof}

We conclude this section with a discussion of the geometry of the Hessenberg Schubert variety~$\Omega_{w_\cs,h}$.
The results of Theorem~\ref{theorem: main} tell us that
\begin{equation}\label{eqn.maintheorem}
\Omega_{w_\cs, h}^T=\Omega_{w_\cs}^T = (\Hess(\mathsf{S}, h) \cap \Omega_{w_\cs})^T
\end{equation}
where the second equality follows from~\cite[Proposition 3]{DeMProSha92}. Since 
\begin{eqnarray}\label{eqn.dimequality}
\dim \Omega_{w_\cs, h}  = \dim  \Omega_{w_\cs, h}^\circ= \dim (\Omega_{w_\cs}\cap \Hess(\mathsf{S},h)) 
\end{eqnarray}
by definition, the equality~\eqref{eqn.maintheorem} might lead one to hope that $\Omega_{w_\cs, h} =\Omega_{w_\cs}\cap \Hess(\mathsf{S},h)$.  This is not true in general.  More specifically, although $\Omega_{w_\cs, h} =\Omega_{w_\cs}\cap \Hess(\mathsf{S},h)$ holds in many cases (see Remark~\ref{rem.equality}), we can conclude in general only that $\Omega_{w_\cs, h}$ is an irreducible component of $\Omega_{w_\cs}\cap \Hess(\mathsf{S},h) = \bigcup_{w\geq w_\cs} \Omega_{w,h}^\circ$. 
We give more details below.

Let $w\in \Symm_n$.  As $\Omega_{w,h}$ is the closure of the affine cell $\Omega_{w, h}^\circ$, it is irreducible.
Now~\eqref{eqn.dimequality} and fact that $\Omega_{w,h} \subseteq \Omega_{w}\cap \Hess(\mathsf{S},h)$ together imply that $\Omega_{w,h}$ is indeed an irreducible component of $\Omega_{w}\cap \Hess(\mathsf{S},h)$.   It now follows that the equality 
\begin{eqnarray}\label{eqn.equal}
\Omega_{w, h} =\Omega_{w}\cap \Hess(\mathsf{S},h)
\end{eqnarray}
holds if and only if $\Omega_{w}\cap \Hess(\mathsf{S},h)$ is irreducible. We consider cases. If $w$ is not the maximal element of $\cw(\cs,h)$ for some $\cs\in \cw_h$, then $\Omega_{w}\cap \Hess(\mathsf{S},h)$ is reducible, since $\Omega_{w, h}^T \subsetneq [w,w_0] = (\Omega_{w}\cap \Hess(\mathsf{S},h))^T$ by Corollary~\ref{cor.reformulation}.  On the other hand, if $w=w_\cs$, the following example  shows that $\Omega_{w_\cs}\cap \Hess(\mathsf{S},h)$ is, in general, still reducible.

\begin{Example}\label{key-ex} 
Let $n=4$ and $h=(3,4,4,4)$. In this case, $\Phi_h^+ = \Phi^+ \setminus \{ t_1-t_4 \}$ and $\dim \Hess(\mathsf{S}, h) = 5$.  Consider $\cs=\{ t_2-t_3, t_1-t_3\} \in \cw_h$.  In Example~\ref{ex.interval} above, it was shown that $w_\cs = [2,3,1,4]$.  We also have
\[
\dim (\Omega_{w_\cs}\cap \Hess(\mathsf{S},h)) = \dim(\Omega_{w_\cs,h}^\circ) = 5-|\cs| = 3 .
\] 
Consider $w=[2,3,4,1]$, which satisfies $w_\cs <w$ so $\Omega_{w,h} \subseteq \Omega_{w_\cs}\cap \Hess(\mathsf{S},h)$. As 
\[
N(w) = \{ t_3-t_4, t_2-t_4, t_1-t_4 \} \Rightarrow \ell_h(w)= |N(w)\cap \Phi_h^+| =2
\]
we have $\dim \Omega_{w,h}^\circ = 3$.  Now the fact that $\Omega_{w_\cs,h}^\circ \cap \Omega_{w,h}^\circ = \emptyset$ and $\dim \Omega_{w,h} = \dim (\Omega_{w_\cs}\cap \Hess(\mathsf{S},h))$ implies $ \Omega_{w,h}$ is an irreducible component of $\Omega_{w_\cs}\cap \Hess(\mathsf{S},h)$ with $\Omega_{w,h}\neq \Omega_{w_\cs, h}$.  This shows $\Omega_{w_\cs}\cap \Hess(\mathsf{S},h)$ is reducible.
\end{Example}

\begin{Remark}\label{rem.equality} When $h = (2,3,4,\ldots, n, n)$, the corresponding regular semisimple Hessenberg variety is the toric variety associated to the Weyl chambers, also called the permutohedral variety.  In this well-studied case, the equality from~\eqref{eqn.equal} holds for all $w_\cs$ with $\cs \in \cw_h$.  This follows from an application of~\cite[Theorem 3.10]{Insko-Precup}. 
\end{Remark}

\appendix
\section{Proof of Lemma~\ref{lemma.interval} by Michael Zeng}\label{sec.appendix}

We now present a proof of Lemma~\ref{lemma.interval}, written by Michael Zeng as part of an undergraduate research project with the second author.  Let $w_0 = [n,n-1,\ldots, 2,1]$ denote the longest element of $\Symm_n$. Consider the map
\[
\varphi: \Symm_n \to \Symm_n, \; \varphi(w) = w_0w.
\]
Since
\begin{eqnarray} \label{eqn.complement}
N(w_0w) = \Phi^+ \setminus N(w)
\end{eqnarray}
for all $w\in \Symm_n$ it follows that $\varphi$ defines an order-reversing involution with respect to the weak order, that is, $u\leq_Lv \Leftrightarrow \varphi(u)\geq_L \varphi(v)$ \cite[Prop.~3.1.5]{BjoBre05}.

\begin{Lemma}\label{lemma.bijection} Let $\cs\in \cw_h$. The restriction of $\varphi$ to $\cw(\cs,h)$ is a 
bijection between $\cw(\mathcal{S}, h)$ and $\cw(\Phi_h^+ \setminus \cs, h)$. 
\end{Lemma}

\begin{proof} Since $\varphi$ is an involution, it suffices to show that $\varphi(w)\in \cw(\Phi_h^+\setminus \cs, h)$ for all $w\in \cw(\cs, h)$.
Intersecting both sides of equation~\eqref{eqn.complement} with $\Phi_h^+$ we obtain $N(w_0w) \cap \Phi_h^+ = \Phi_h^+ \setminus (N(w) \cap \Phi_h^+)$, which is equivalent to $N(w_0w) \cap \Phi_h^+ = \Phi_h^+ \setminus \mathcal{S}$. Thus $\varphi(w) = w_0w\in W(\Phi_h^+ \setminus \mathcal{S}, h)$, as desired.  
\end{proof}

Using the previous lemma and Theorem~\ref{thm.Sommers-Tymoczko}, we obtain the following corollary.

\begin{Corollary}\label{cor.max} For each $\cs\in \cw_h$, 
$\cw(\mathcal{S}, h)$ has a unique maximal element with respect to weak Bruhat order. 
\end{Corollary}

\begin{proof}
Since $\cs$ is a subset of Weyl type with respect to $h$, the complement $\bar{\cs} = \Phi_h^+ \setminus \mathcal{S}$ is also a subset of Weyl type with respect to $h$. 
By Theorem~\ref{thm.Sommers-Tymoczko}(3), the set $\cw(\Phi_h^+ \setminus \cs,h)$ has a unique minimal element with respect to the weak Bruhat order, denoted $z_{\bar{\cs}}$. By Lemma~\ref{lemma.bijection}, $\varphi_{\bar{\cs}}(z_{\bar{\cs}})\in \cw(\cs, h)$. 
Since $\varphi$ is an order reversing involution and $z_{\bar{\cs}}$ is the unique minimal element in $\cw(\Phi_h^+ \setminus \cs,h)$, we conclude that $\varphi(z_{\bar{\cs}})$ is the unique maximal element in $W(\cs,h)$.
\end{proof}

Finally, we prove Lemma~\ref{lemma.interval}: $\cw(\cs, h)$ is an interval in the weak Bruhat order. 

\begin{proof}[Proof of Lemma~\ref{lemma.interval}] 
Let $z_\cs$ and $w_\cs$ be the unique minimum and maximum elements, respectively, of $\cw(\cs,h)$ with respect to the weak (left) Bruhat order.  Such elements must exist by Theorem~\ref{thm.Sommers-Tymoczko}(3) and Corollary~\ref{cor.max}. The inclusion $\cw(\cs, h) \subseteq [z_\cs, w_\cs]_L$ is immediate. To show the converse, let $w\in [z_\cs, w_\cs]_L$. This means $z_\cs \leq_L w \leq_L w_\cs$, or equivalently, $N(z_\cs) \subseteq N(w) \subseteq N(w_\cs)$. Taking the intersection with $\Phi_h^+$ at each term in this chain, we obtain $N(z_\cs) \cap \Phi_h^+ \subseteq N(w)\cap \Phi_h^+ \subseteq  N(w_\cs)\cap \Phi_h^+$. Since $z_\cs, w_\cs\in \cw(\cs,h)$ we have $N(z_\cs) \cap \Phi_h^+ = N(w_\cs) \cap \Phi_h^+ = \cs$, implying $N(w)\cap \Phi_h^+ = \cs$. This proves $w \in \cw(\cs, h)$, as desired.
\end{proof}

\newcommand{\etalchar}[1]{$^{#1}$}
\def\cprime{$'$}

{\ifdebug {\newpage} \else {\end{document}} \fi}


\begin{thebibliography}{H}  


%
%
%
%

\bibitem{AbeHoriguchi} 
H. Abe and T.~Horiguchi. 
\newblock A survey of recent developments on Hessenberg
varieties. 
\newblock {\em Schubert calculus and its applications 
in combinatorics and representation theory}, 251-279, 
Springer Proc. Math. Stat., 332, Springer, 2020. 
%
%
%
\bibitem{BjoBre05}
A. Bj{\"o}rner and F. Brenti.
\newblock {\em Combinatorics of {C}oxeter groups}, vol. 231, {\em Graduate
  Texts in Mathematics}. Springer, 2005.
%
%
%
%
%
\bibitem{Cho-Hong-Lee1}
S. Cho, J. Hong, and E. Lee.
\newblock Bases of the equivariant cohomologies of regular semisimple
  Hessenberg varieties, ArXiv: 2008.12500, 2020.

\bibitem{Cho-Hong-Lee2}
Soojin Cho, Jaehyun Hong, and Eunjeong Lee.
\newblock Permutation module decomposition of the second cohomology of a
  regular semisimple Hessenberg variety, ArXiv: 2107.00863, 2021.
%
\bibitem{DeMProSha92}
F.~De~Mari, C.~Procesi, and M.~A. Shayman.
\newblock Hessenberg varieties.
\newblock {\em Trans. Amer. Math. Soc.}, 332(2):529--534, 1992.
%
%

\bibitem{Harada-Precup2019}
M. Harada and M. Precup.
\newblock The cohomology of abelian {H}essenberg varieties and the
  {S}tanley-{S}tembridge conjecture.
\newblock {\em Algebr. Comb.}, 2(6):1059--1108, 2019.
%
\bibitem{Humphreys} 
J. Humphreys. 
\newblock Reflection groups and Coxeter groups.
\newblock {\em Cambridge Studies in Advanced Mathematics} 
volume 29, Cambridge University Press, 1990. 
%
\bibitem{Insko-Precup}
E.~Insko and M.~Precup.
\newblock The singular locus of semisimple {H}essenberg varieties.
\newblock {\em J.~Algebra}, 521:55--96, 2019.


\bibitem{Kos61}
B. Kostant.
\newblock Lie algebra cohomology and the generalized {B}orel-{W}eil theorem.
\newblock {\em Ann. of Math. (2)}, 74:329--387, 1961.

%
%

\bibitem{Precup2018}
M. Precup.
\newblock The {B}etti numbers of regular {H}essenberg varieties are
  palindromic.
\newblock {\em Transf. Groups}, 23(2):491--499, 2018.
%
%
\bibitem{SomTym06}
E. Sommers and J. Tymoczko.
\newblock Exponents for {$B$}-stable ideals.
\newblock {\em Trans. Amer. Math. Soc.}, 358(8):3493--3509, 2006.
%
\bibitem{ShareshianWachs2016}
J. Shareshian and M. Wachs.
\newblock Chromatic quasisymmetric functions.
\newblock {\em Adv. Math.}, 295:497--551, 2016.

%

\end{thebibliography}
\end{document}